\documentclass[12pt,reqno]{article}

\usepackage[usenames]{color}
\usepackage{amssymb}
\usepackage{amsmath}
\usepackage{amsthm}
\usepackage{amsfonts}
\usepackage{amscd}
\usepackage{graphicx}

\usepackage[colorlinks=true,
linkcolor=webgreen,
filecolor=webbrown,
citecolor=webgreen]{hyperref}

\definecolor{webgreen}{rgb}{0,.5,0}
\definecolor{webbrown}{rgb}{.6,0,0}

\usepackage{color}
\usepackage{fullpage}
\usepackage{float}

\usepackage{graphics}
\usepackage{latexsym}
\usepackage{epsf}
\usepackage{breakurl}

\newcommand{\seqnum}[1]{\href{https://oeis.org/#1}{\rm \underline{#1}}}

\def\Enn{\mathbb{N}}
\DeclareMathOperator{\evil}{evil}

\DeclareMathOperator{\inc}{inc}

\begin{document}

\theoremstyle{plain}
\newtheorem{theorem}{Theorem}
\newtheorem{corollary}[theorem]{Corollary}
\newtheorem{lemma}[theorem]{Lemma}
\newtheorem{proposition}[theorem]{Proposition}

\theoremstyle{definition}
\newtheorem{definition}[theorem]{Definition}
\newtheorem{example}[theorem]{Example}
\newtheorem{conjecture}[theorem]{Conjecture}

\theoremstyle{remark}
\newtheorem{remark}[theorem]{Remark}

\title{Frobenius Numbers and Automatic Sequences}

\author{Jeffrey Shallit\footnote{Research funded by a grant from NSERC, 2018-04118.}\\
School of Computer Science\\
University of Waterloo\\
Waterloo, ON  N2L 3G1\\
Canada\\
\href{mailto:shallit@uwaterloo.ca}{\tt shallit@uwaterloo.ca}}

\maketitle

\begin{abstract}
The Frobenius number $g(S)$
of a set $S$ of non-negative integers with $\gcd 1$ is
the largest integer not expressible as a linear combination of elements
of $S$.  Given a sequence ${\bf s} = (s_i)_{i \geq 0}$,
we can define the associated sequence
$G_{\bf s} (i) = g(\{ s_i,s_{i+1},\ldots \})$.  In this paper
we compute $G_{\bf s} (i)$ for some classical automatic sequences:
the evil numbers, the odious numbers, and the lower and upper Wythoff
sequences.   In contrast with the usual methods, our proofs are
based largely on automata theory and logic.

\bigskip

\end{abstract}

\section{Introduction}

Let $\Enn = \{ 0,1,2, \ldots \}$ be the natural numbers.
Let $S$ be a nonempty set of natural numbers with $\gcd(S) = 1$,
possibly infinite.
A classical result then says that every sufficiently large integer
can be written as a linear combination of
elements of $S$ with natural number coefficients.
Then $g(S)$, the {\it Frobenius number\/} of $S$,
is defined to be the greatest integer $t$
such that $t$ does {\it not\/} have such a representation.
For example, $g(\{ 6,9,20 \}) = 43$.

The Frobenius number has received a lot of attention in recent
years.
For a detailed discussion of the function $g$, see the books
of Ram{\'i}rez Alfons{\'i}n \cite{RamirezAlfonsin05} and
Rosales and Garc{\'i}a-S{\'a}nchez \cite{Rosales&Garcia-Sanchez:2009}.  

Let ${\bf s} = (s_i)_{i \geq 0}$ be
an increasing sequence of natural numbers such that 
$\gcd(s_i, s_{i+1}, \ldots) = 1$ for all $i \geq 0$.  
For $i \geq 0$ define $G_{\bf s} (i) = g(\{ s_i, s_{i+1}, \ldots \})$, the
Frobenius number of a final segment of $\bf s$, beginning with $s_i$.
Given a sequence ${\bf s}$, it can be an interesting and challenging problem
to compute the Frobenius numbers
$G_{\bf s}(i)$ exactly, or estimate their growth rate.
Computing the Frobenius number is notoriously difficult
because the problem is NP-hard \cite{RamirezAlfonsin96}.

For example, for 
the case where ${\bf s}$ is $1,4,9,16, \ldots$, the sequence of
squares, Dutch and Rickett
\cite{Dutch&Rickett:2012} proved that
$G_{\bf s} (n) = o(n^{2+\epsilon})$ for
all $\epsilon > 0$, and this bound was improved
to $O(n^2)$ by Moscariello \cite{Moscariello:2015}.

In this paper we are interested in calculating $G_{\bf s} (i)$ for
some famous integer sequences $\bf s$ whose characteristic sequence
is automatic.   (The characteristic sequence is the binary
sequence ${\bf b} = (b_i)_{i \geq 0}$ 
where $b_i = 1$ if $i$ occurs in $\bf s$ and $b_i = 0$
otherwise.)
By ``automatic'' we mean 
generated
by a finite automaton in a certain way, described in the next
section.   In particular, we
completely characterize the Frobenius number
for the sequence of evil numbers, the sequence of odious
numbers, and the lower and upper Wythoff sequences.
Although our results are number-theoretic in nature,
our approach is largely via automata theory and logic.

\section{Automatic and synchronized sequences}

A {\it numeration system\/} is a method for writing a non-negative
integer $N$ as a linear combination $N = \sum_{0 \leq i \leq t} a_i d_i$
of some increasing sequence
$d_0 = 1 < d_1 < d_2 < \cdots $ of integers, where the
$a_i$ are chosen from a finite set that includes $0$.
Suppose
\begin{itemize}
\item[(a)]  the representation for $N$ is always unique (up to leading zeros);
\item[(b)]  the set of all valid representations forms a regular language;
\item[(c)]  there is a finite automaton recognizing the triples
$(x,y,z)$ for which $x+y = z$, where the inputs to
the automaton are representations
of integer triples, padded with leading zeros, if necessary, to make them
all of the same length.
\end{itemize}
If all three conditions hold, then we call the numeration system {\it regular}.
A sequence ${\bf a} = (a_i)_{i \geq 0}$ over a finite alphabet
is {\it automatic} if it is computed by an automaton taking the
representation of $i$ as input and returning $a_i$ as the output
associated with the last state reached.  For more about automatic
sequences, see \cite{Allouche&Shallit:2003}.

If a sequence is automatic, then there is a decision procedure for
proving or disproving
assertions about it, provided these assertions are phrased
in first-order logic and using only the logical operations, comparisons
of natural numbers, addition, and indexing into the sequence
\cite{Charlier&Rampersad&Shallit:2012}.

Automatic sequences are restricted in scope because they take their values
in a finite alphabet.  However, it's possible for automata to compute
sequences taking their values in $\Enn$, the natural numbers, using
a different meaning of ``compute''.   We say
a sequence $(a_i)_{i \geq 0}$ is {\it synchronized\/} if there is
a finite automaton recognizing exactly the
representations of pairs $(i, a_i)$, read in parallel.
If a sequence is synchronized, then we can
compute it in linear time.   If the underlying numeration system is
base $k$, then the quantities $\limsup_{i} a_i/i$ and
$\liminf_{i} a_i/i$ are computable \cite{Schaeffer&Shallit:2012}.
For more about synchronized sequences, see
\cite{Carpi&Maggi:2001}.

\section{The odious and evil numbers}

Let $s_2 (n)$ denote the sum of bits of $n$ when represented in
base $2$.   Define $t(n) = s_2 (n) \bmod 2$, the famous
Thue-Morse sequence \cite{Allouche&Shallit:1999}.
The {\it evil numbers}
$0, 3,5,6,9, \ldots$  are those $n$ with $t(n) = 0$,
and the {\it odious numbers} $1,2,4,7,\ldots$ are those $n$ with $t(n) = 1$.
(The somewhat painful terminology is from
\cite[p.~431]{Berlekamp&Conway&Guy:1982}.)
More precisely we have \cite{Allouche&Cloitre&Shevelev:2016}
\begin{align*}
e_n &= 2n + t(n) \\
o_n &= 2n + 1-t(n)
\end{align*}
for $n \geq 0$.
Observe that both sequences obey the $\gcd$ condition,
because $\gcd(e_i, e_{i+1}, e_{i+2}) = \gcd(o_i, o_{i+1}, o_{i+2}) = 1$
for all $i\geq 0$.  Additive properties of these numbers were
studied previously in \cite[Thm.~2]{Rajasekaran&Shallit&Smith:2020}.

Our first goal is to prove the following result:
\begin{theorem}
The function $G_{\bf e} (n)$ is $2$-synchronized.
\label{evilsynch}
\end{theorem}

In order to prove this we need the following lemmas:

\begin{lemma}
If $n \geq 1$ can be written as the sum of
four evil numbers $\geq m$, then $n$ can be written as the sum of
one, two, or three evil numbers $\geq m$.
\label{one}
\end{lemma}

\begin{remark}
When we speak of such sums we never insist that the sum be of distinct
integers.
\end{remark}

\begin{proof}
This assertion can be phrased as a first-order formula, namely,
$$\forall m, n \ \evil_{4}(m,n) \Longrightarrow \evil_{1,2,3}(m,n) , $$
where 
\begin{align*}
\evil_{1,2,3}(m,n) & := \exists j, k, \ell \ t(j)=t(k)=t(\ell) = 0 
\, \wedge\, j,k,\ell \geq m \, \wedge\, (n=j \, \vee \, n=j+k \, \vee\, 
n=j+k+\ell) \\
\evil_{4}(m,n) &:= \exists i, j, k, \ell \ t(i)=t(j)=t(k)=t(\ell)=0 
\, \wedge\, i,j,k,\ell \geq m \, \wedge\, n=i+j+k+\ell .
\end{align*}

To prove this, we use the theorem-proving software called {\tt Walnut},
where we can simply translate the statement of the previous
paragraph into {\tt Walnut}'s syntax and evaluate it.
\begin{verbatim}
def evil123rep "Ej,k,l (T[j]=@0) & (T[k]=@0) & (T[l]=@0) &
     j>=m & k>=m & l>=m & (n=j | n=j+k | n=j+k+l)":

def evil4rep "Ei,j,k,l (T[i]=@0) & (T[j]=@0) & (T[k]=@0) & 
(T[l]=@0) & i >= m & j>=m & k>=m & l>=m & (n=j+k+l+m)":

eval evilcheck "Am,n $evil4rep(m,n) => $evil123rep(m,n)":
\end{verbatim}
This returns {\tt true}, so the lemma is proved.
\end{proof}

\begin{remark}
To help in understanding the syntax of {\tt Walnut},
we note the following:
\begin{itemize}
\item {\tt E} represents $\exists$
\item {\tt A} represents $\forall$
\item {\tt T} represents the Thue-Morse
sequence
\item the natural number constant $i$ is written {\tt @i}
\item {\tt \&} represents logical ``and''
\item {\tt |} represents logical ``or''
\item {\tt \char'176} represents logical negation
\item {\tt =>} is logical implication
\item {\tt def} defines a formula for future use
\item {\tt eval} evaluates a logical formula and returns true or false.
\end{itemize}
\end{remark}

\begin{lemma}
Let $n \geq 1$ be a non-negative integer linear combination of
evil numbers $\geq m$.   Then $n$ can be written as the
sum of either one, two, or three evil numbers $\geq m$ .
\label{two}
\end{lemma}

\begin{proof}
Let $n$ be written as a non-negative
integer linear combination of evil numbers $\geq m$.
Without loss of generality choose a representation for $n$
that minimizes $s$, the sum of the coefficients.   If this sum
is at least $4$, we can write $n = u+v$
where $u$ is the sum of $4$ evil numbers and $v$ is the
sum of $n-4$ evil numbers, all $\geq m$.
But then by Lemma~\ref{one}, we can write 
$u$ as the sum of $3$ evil numbers $\geq m$, so $n$ is
the sum of $n-1$ evil numbers $\geq m$, a contradiction.
So $s$ is no more than $3$, as desired.
\end{proof}

We can now prove Theorem~\ref{evilsynch}.
\begin{proof}
It suffices to give a first-order definition of $G_{\bf e} (n)$.
We can do this as follows:
\begin{verbatim}
def evilg "(Aj (j>n) => $evil123rep(2*m,j)) & ~$evil123rep(2*m,n)":
\end{verbatim}
This gives a $58$-state synchronized automaton computing $G_{\bf e} (n)$.
\end{proof}

Now that we have a synchronized automaton, we can determine 
the asymptotic behavior of $G_{\bf e} (n)$.

\begin{theorem}
We have $4m \leq G_{\bf e} (m) \leq 6m+7$ for all $m \geq 0$.  
These bounds are optimal, because they are attained 
for infinitely many $m$.
\end{theorem}

\begin{proof}
We run the following {\tt Walnut} commands, which all evaluate to
{\tt true}.
\begin{verbatim}
eval upperb "Am,n $evilg(m,n) => n <= 6*m+7":
eval upperopt "Ai Em,n (m>i) & $evilg(m,n) & n=6*m+7":
eval lowerb "Am,n $evilg(m,n) => n >= 4*m":
eval loweropt "Ai Em,n (m>i) & $evilg(m,n) & n=4*m":
\end{verbatim}
\end{proof}

\begin{corollary}
We have
\begin{align*}
\inf_{i \geq 1} G_{\bf e} (i)/i &= 4  & 
\sup_{i \geq 1} G_{\bf e} (i)/i &= 7  \\
\liminf_{i \geq 1} G_{\bf e} (i)/i &= 4 &
\limsup_{i \geq 1} G_{\bf e} (i)/i &= 6 .
\end{align*}
\end{corollary}

The sequence $(G_{\bf e} (i))_{i \geq 0}$ has a rather erratic
behavior.  In particular we can prove 
\begin{theorem}
\leavevmode
\begin{itemize}
\item[(a)]  The difference $G_{\bf e} (i+1) - G_{\bf e} (i)$ can
be arbitrarily large.
\item[(b)]  There are arbitrarily long blocks of indices on
which $G_{\bf e}(i)$ is constant.
\end{itemize}
\label{blocks}
\end{theorem}

\begin{proof}
We use the following {\tt Walnut} code.
\begin{verbatim}
eval evilgdiff "Ai Ej,m,n1,n2 (j>=i) & $evilg(m,n1) & $evilg(m+1,n2) & n2=n1+j":
eval evalmonotone "Ai Ej,m,u (j>=i) & $evilg(m,u) &
   (At,v ((t>m) & (t<m+j) & $evilg(t,v)) => u=v)":
\end{verbatim}
Both return {\tt true}.
\end{proof}

We can carry out exactly the same analysis for the odious numbers.
The analogues of Theorem~\ref{evilsynch}
and Lemma~\ref{one} and \ref{two} all hold.
Here are the results:
\begin{theorem}
We have $4m \leq G_{\bf o} (m) \leq 6m-1$ for all $m \geq 1$.
These bounds are optimal, because they are attained
for infinitely many $m$.
\end{theorem}

\begin{corollary}
We have
\begin{align*}
\inf_{i \geq 1} G_{\bf o} (i)/i &= 4  &
\sup_{i \geq 1} G_{\bf o} (i)/i &= 6 \\
\liminf_{i \geq 1} G_{\bf o} (i)/i &= 4 &
\limsup_{i \geq 1} G_{\bf o} (i)/i &= 6 .
\end{align*}
\end{corollary}

Furthermore, the analogue of Theorem~\ref{blocks} also holds.

The first few terms of the sequences we have discussed in this
section, together with their numbers from the {\it On-Line
Encyclopedia of Integer Sequences} (OEIS) \cite{Sloane}, are given below.

\begin{center}
\begin{tabular}{|c|c|rrrrrrrrrrrrrrrr}
 OEIS& \\
 number& $n$ & 0& 1& 2& 3& 4& 5& 6& 7& 8& 9&10&11&12&13&14&15 \\
\hline
\seqnum{A001969} & $e_n$ & 0& 3& 5& 6& 9&10&12&15&17&18&20&23&24&27&29&30 \\
\seqnum{A342581} & $G_{\bf e} (n)$ & 7& 7&13&14&16&31&31&31&32&55&55&55&55&55&61&62\\
\seqnum{A000069} & $o_n$ & 1& 2& 4& 7& 8&11&13&14&16&19&21&22&25&26&28&31\\
\seqnum{A342579}& $G_{\bf o} (n)$ & $-1$& 5&10&17&23&23&24&34&39&39&45&46&71&71&71&71 
\end{tabular}
\end{center}

\section{Results for the Wythoff sequences}

We can carry out a similar analysis for the lower and upper
Wythoff sequences, defined as follows.  Here we find substantially
different behavior than for the odious and evil numbers.

Let $\varphi = (1+\sqrt{5})/2$, the golden ratio.
Define
\begin{align*}
L_n &= \lfloor n \varphi \rfloor \\
U_n &= \lfloor n \varphi^2 \rfloor 
\end{align*}
for $n \geq0$.
Here, instead of base-$2$ representation, all numbers are
represented in Fibonacci representation (also called
Zeckendorf representation) \cite{Lekkerkerker:1952,Zeckendorf:1972}.
In this representation a number is represented as a linear
combination $\sum_{2 \leq i \leq t} a_i F_i$ with $a_i \in \{ 0, 1 \}$
and subject to the condition that $a_i a_{i+1} = 0$.  We
then define $(n)_F = a_t a_{t-1} \cdots a_2$ and
$[x]_F = \sum_{1 \leq i \leq t} a_i F_{t+2-i}$  if
$x = a_1 \cdots a_t$.

The additive properties of the upper and lower Wythoff
sequences were studied previously in \cite{Kawsumarng&Khemaratchatakumthorn&Noppakaew&Pongsriiam:2021,Shallit:2021}.

\begin{theorem}
The functions $L_n$ and $U_n$ are Fibonacci synchronized.
\end{theorem}

\begin{proof}
We start by showing that the function $n+1$ is Fibonacci synchronized.
We can construct an automaton $\inc(x,y)$
that computes the relation $y = x + 1$ for $x$ and $y$ in
Fibonacci representation.  
Using the following easily-proven identities,
\begin{itemize}
\item[(a)] $[x00(10)^i]_F + 1 = [x01 0^{2i}]_F$; and
\item[(b)] $[x0(01)^i]_F + 1 = [x01 0^{2i-1}]_F$,
\end{itemize}
we can obtain the incrementer depicted in Figure~\ref{increm} below.
\begin{figure}[H]
\begin{center}
\includegraphics[width=4.5in]{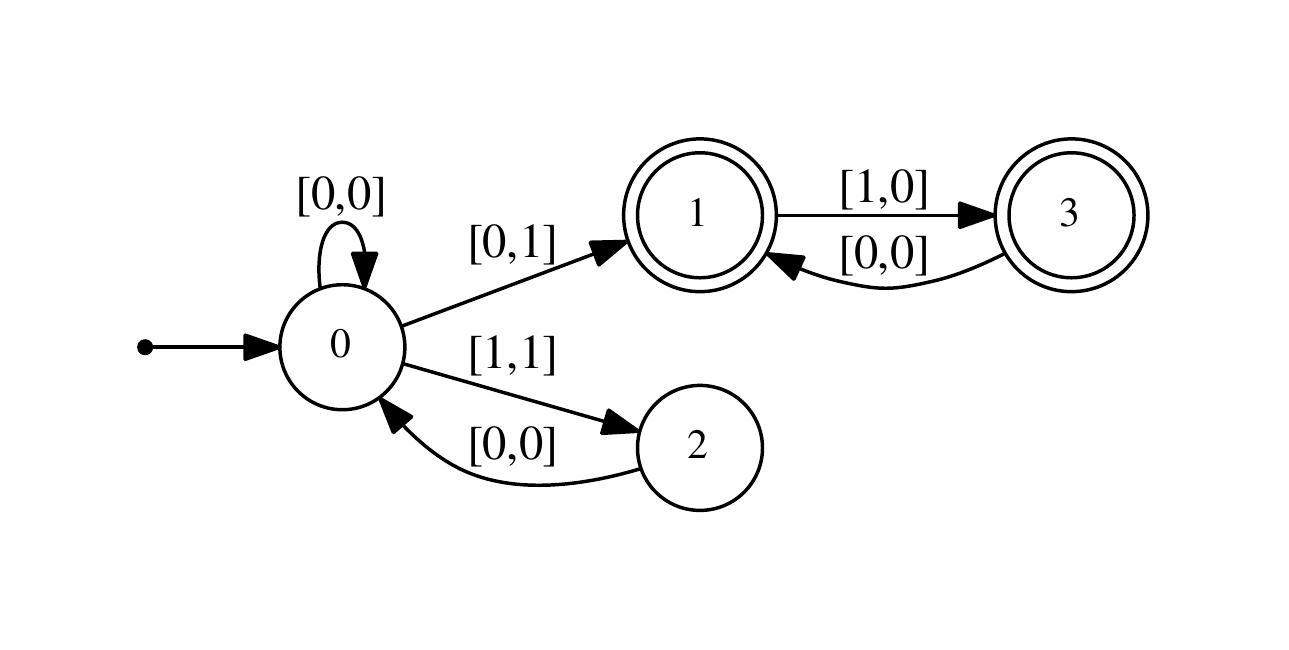}
\end{center}
\label{increm}
\caption{Incrementer for Fibonacci representation}
\end{figure}

Next, we use the identities
\begin{align*}
[(n)_F 0]_F &= \lfloor (n+1) \alpha \rfloor - 1 \\
[(n)_F 01]_F &= \lfloor (n+1) \alpha^2 \rfloor - 1
\end{align*}
for $n \geq 0$, whose proof can be found, for example, in 
\cite{Reble:2008}.   
Substituting $n-1$ for $n$, this gives us formulas for $\lfloor n
\alpha \rfloor$ and $\lfloor n \alpha^2 \rfloor$ in terms of
shifting and incrementation in the Fibonacci representation.
Shifting can be carried out using the Fibonacci synchronized DFA below:
\begin{figure}[H]
\begin{center}
\includegraphics[width=4in]{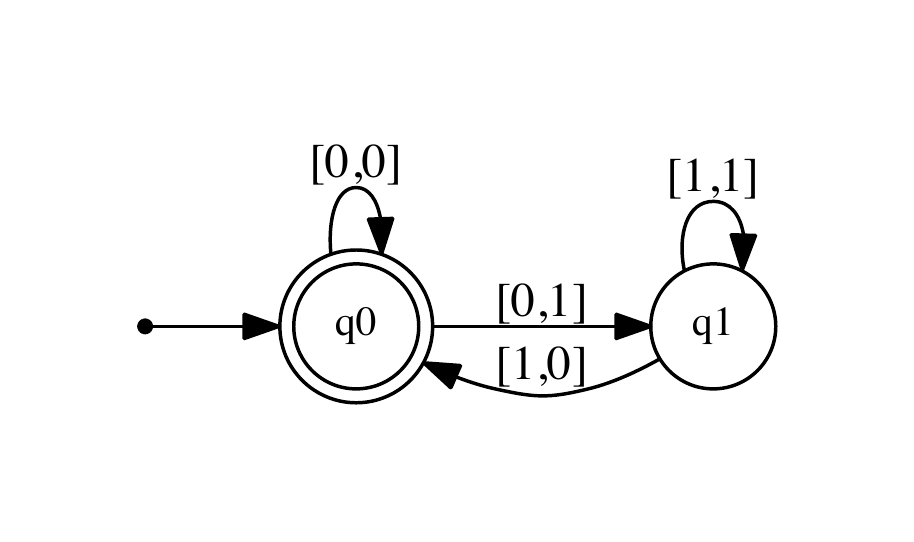}
\end{center}
\label{shif}
\caption{Shifter for Fibonacci representation}
\end{figure}

So we get synchronized automata for (a) and (b) as follows:
\begin{verbatim}
def fibna "?msd_fib ((s=0)&(n=0)) | Et,u $fibinc(u,n) & $shift(u,t) &
     $fibinc(t,s)":
def fibna2 "?msd_fib ((s=0)&(n=0)) | Et,u,v,w $fibinc(u,n) & $shift(u,t)
     & $shift(t,v) & $fibinc(v,w) & $fibinc(w,s)":
\end{verbatim}
with the automata depicted below in Figure~\ref{fibinc}:
\begin{figure}[ht]
\begin{center}
\includegraphics[width=6in]{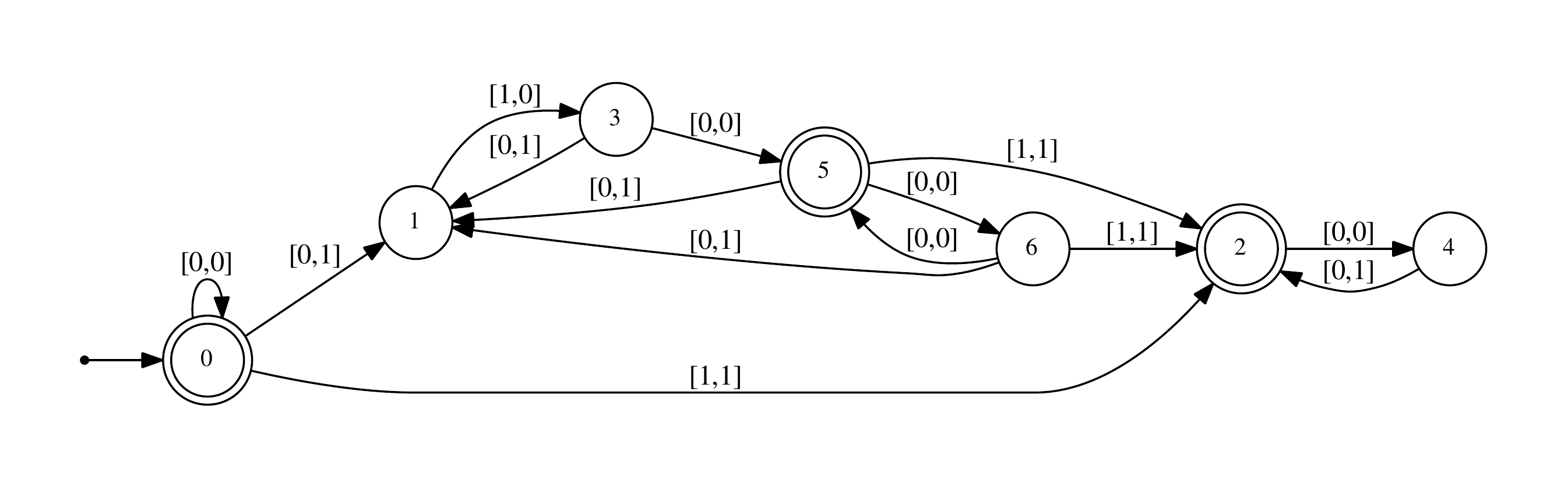}
\includegraphics[width=6in]{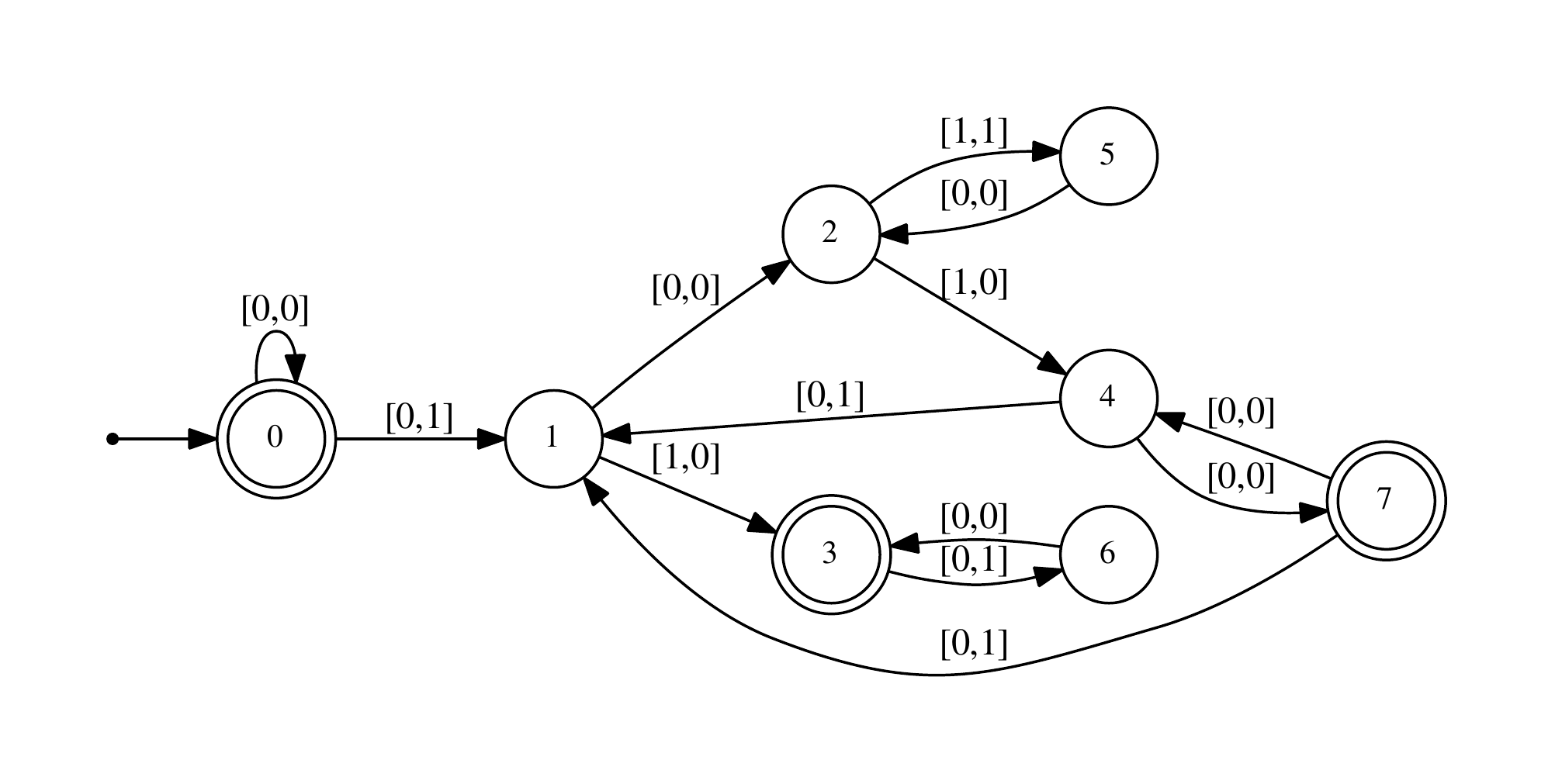}
\end{center}
\caption{Synchronized automata for $\lfloor n \alpha \rfloor$ (top) and
$\lfloor n \alpha^2 \rfloor$ (bottom).}
\label{fibinc}
\end{figure}
\end{proof}

Our next goal is to prove that $G_{\bf L} (n)$ is Fibonacci-synchronized
(that is, there is an automaton recognizing the pairs $(n, G_{\bf L} (n))$
represented in the Fibonacci numeration system).   We start with
some lemmas.

\begin{lemma}
Let $n \geq 1$
be a non-negative integer linear combination of lower Wythoff
numbers $\geq m$.
Then $n$ can be written as the sum of either one or two
lower Wythoff numbers $\geq m $.
\end{lemma}

\begin{proof}
It can be carried out with {\tt Walnut} in analogy with the proof
of Lemma~\ref{two}.
\begin{verbatim}
def lower12rep "?msd_fib Ej,k (F[j-1]=@0) & (F[k-1]=@0) & 
     j>=m & k>=m & (n=j|n=j+k)":
def lower3rep "?msd_fib Ej,k,l (F[j-1]=@0) & (F[k-1]=@0) & (F[l-1]=@0) &
     j>=m & k>=m & l>=m & n=j+k+l":
eval fibcheck "?msd_fib Am An $lower3rep(m,n) => $lower12rep(m,n)":
\end{verbatim}
\end{proof}

We can now prove
\begin{theorem}
The function $G_{\bf L} (n)$ is Fibonacci synchronized.
\end{theorem}
\begin{proof}
We use the {\tt Walnut} commands:
\begin{verbatim}
def lowerunrep "?msd_fib (Aj (j>n) => $lower12rep(m,j)) & ~$lower12rep(m,n)":
def lowerg "?msd_fib Et $fibna(m,t) & $lowerunrep(t,n)":
\end{verbatim}
This gives us a synchronized automaton with $24$ states computing $G_{\bf L} (n)$.
\end{proof}

We can now use this automaton to determine the behavior of
$G_{\bf L} (n)$.
\begin{theorem}
We have $-3 \leq G_{\bf L} (n) - 2 L_n \leq 1$, and the upper
and lower bounds are achieved infinitely often.
\end{theorem}

\begin{proof}
We use the following {\tt Walnut} commands:
\begin{verbatim}
eval lowerb1 "?msd_fib Am En,r $lowerg(m,n) & $fibna(m,r) & n+3>=2*r":
eval lowerb2 "?msd_fib Am En,r $lowerg(m,n) & $fibna(m,r) & n<=2*r+1":
eval lowerbinf1 "?msd_fib As Em,n,r (m>=s) & $lowerg(m,n) & 
     $fibna(m,r) & n+3>=2*r":
eval lowerbinf2 "?msd_fib As Em,n,r (m>=s) & $lowerg(m,n) &
     $fibna(m,r) & n<=2*r+1":
\end{verbatim}
\end{proof}

\begin{corollary}
We have $\lim_{n \rightarrow \infty} G_{\bf L} (n)/n = 1 +\sqrt{5}$.
\label{nine}
\end{corollary}

\begin{theorem}
We have $G_{\bf L} (n+1) - G_{\bf L} (n) \in \{ 0,2,3,5,6,8 \}$, 
and furthermore each difference occurs infinitely often.
\end{theorem}

\begin{proof}
We use the following {\tt Walnut} commands:
\begin{verbatim}
eval lowerdiff "?msd_fib Am Eu,v $lowerg(m,u) & $lowerg(m+1,v) & 
      (v=u|v=u+2|v=u+3|v=u+5|v=u+6|v=u+8)":
def ldi "?msd_fib Am Et,u,v (t>=m) & $lowerg(t,u) & $lowerg(t+1,v) & v=u+d":
eval lowerdiffinfcheck "?msd_fib $ldi(0) & $ldi(2) & $ldi(3) & $ldi(5) 
      & $ldi(6) & $ldi(8)":
\end{verbatim}
\end{proof}

\begin{theorem}
There exists a Fibonacci automaton of $11$ states
computing the first difference
$G_{\bf L}(n+1)-G_{\bf L}(n)$.
\end{theorem}
The automaton is depicted in Figure~\ref{aut11} below.
\begin{figure}[H]
\begin{center}
\includegraphics[width=6.5in]{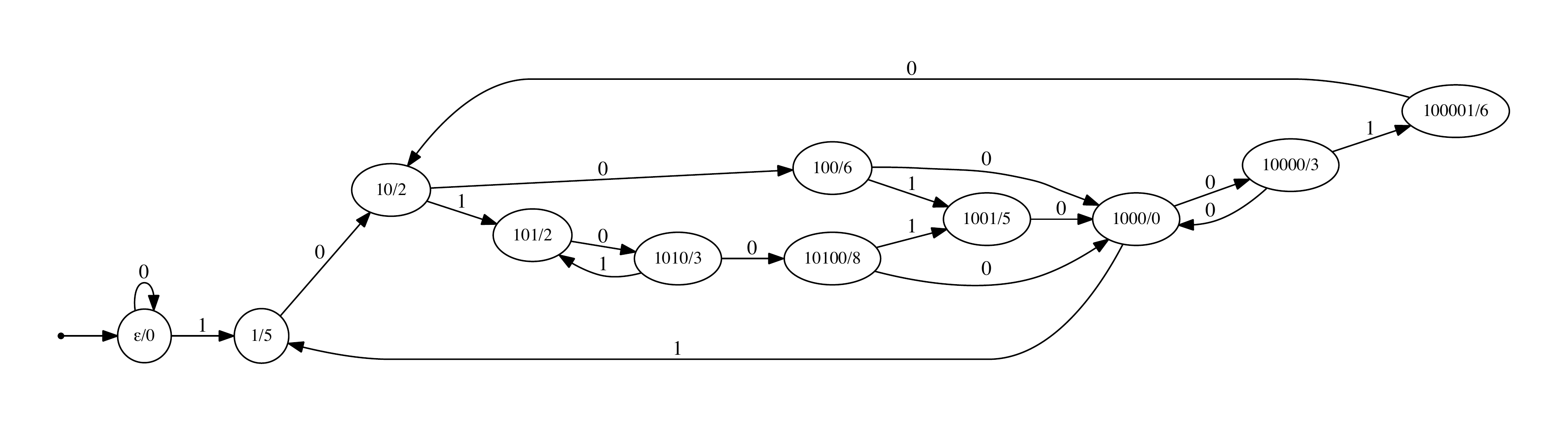}
\end{center}
\caption{Fibonacci automaton computing $G_{\bf L}(n+1)-G_{\bf L}(n)$}
\label{aut11}
\end{figure}

Now we turn to the upper Wythoff sequence.   The results are
completely analogous to the results for the lower Wythoff sequence, and
the proofs are also analogous.  We omit the details.

\begin{lemma}
Let $n \geq 1$
be a non-negative integer linear combination of upper Wythoff
numbers $\geq m$, for $m \geq 3$.
Then $n$ can be written as the sum of either one, two, or
three Wythoff numbers $\geq m $.
\label{l16}
\end{lemma}

\begin{remark}
Lemma~\ref{l16} fails for $m = 2$ because
$8$ is not the sum of one, two, or three Wythoff numbers $\geq 2$,
while it is the sum of four (since $8=2+2+2+2$).
\end{remark}

\begin{theorem}
The function $G_{\bf U} (n)$ is Fibonacci synchronized.
\end{theorem}

\begin{theorem}
We have $-5 \leq G_{\bf U} (n) - 3 U_n \leq 20$, and these
upper and lower bounds are achieved infinitely often.
\end{theorem}

\begin{corollary}
We have $\lim_{n \rightarrow \infty} G_{\bf U} (n)/n = (9+3 \sqrt{5})/2$.
\end{corollary}

\begin{theorem}
We have $G_{\bf U} (n+1) - G_{\bf U} (n) 
\in \{ 0,3,5,8,11,13,18,21,23,26,31 \}$ for $n \geq 1$, 
and furthermore each difference occurs infinitely often.
\end{theorem}

\begin{theorem}
There exists a Fibonacci automaton 
computing the first difference
$G_{\bf U}(n+1)-G_{\bf U}(n)$.
\end{theorem}

The first few terms of the sequences we have discussed in this
section, together with their numbers from the OEIS, are given below.

\begin{center}
\begin{tabular}{|c|c|rrrrrrrrrrrrrrrrr}
 OEIS& \\
 number& $n$ & 0& 1& 2& 3& 4& 5& 6& 7& 8& 9&10&11&12&13&14 \\
\hline
\seqnum{A000201} & $L_n$ & 0&  1&  3&  4&  6&  8&  9& 11& 12& 14& 16& 17& 19& 21& 22 \\
\seqnum{A342715} & $G_{\bf L} (n)$ & $-1$& $-1$&  5&  7& 13& 15& 15& 20& 23& 26& 31& 31& 39& 41& 41 \\
\seqnum{A001950} & $U_n$ & 0&  2&  5&  7& 10& 13& 15& 18& 20& 23& 26& 28& 31& 34& 36 \\
\seqnum{A342716} & $G_{\bf U} (n)$ & $-1$&  3& 16& 19& 42& 42& 42& 55& 58& 76& 79& 79&110&110&110
\end{tabular}
\end{center}

\section{A counterexample}

In all of the examples we have seen so far, if a sequence had automatic
characteristic sequence, then the characteristic sequence of the associated
Frobenius sequence was also automatic.   It is natural to conjecture this
might always be the case.  However, we now prove
\begin{theorem}
Let $s_i = 2^i + 1 $ for $i \geq 0$
and ${\bf s} = (s_i)_{i \geq 0}$. 
Then $G_{\bf s} (i) = 2^{2i} + 2^i + 1$ for $i \geq 1$.
\label{counter}
\end{theorem}

\begin{proof}
It suffices to prove that $2^{2i} + 2^i + 1$ cannot be written as a 
non-negative integer linear combination of $2^i+1,2^{i+1}+1,\ldots 2^{2i}+1$,
while every larger integer can be so expressed.

Suppose $2^{2i} + 2^i + 1 = a_0 (2^i+1) + a_1 (2^{i+1} +1) + 
\cdots + a_i (2^{2i} + 1)$ with $a_0, \ldots, a_i $ non-negative integers.
Considering both sides modulo $2^i$, we see
the left-hand side is $1$, while the right-hand side is $a_0+a_1+\cdots + a_i$.
So either $a_0+a_1+\cdots+a_i =1$ or $a_0+a_1+\cdots+a_i \geq 2^i + 1$.
In the former case we would have $2^{2i} +2^i +1 = 2^j + 1$ for
some $j$, $i\leq j \leq 2i$, which is clearly impossible.
In the latter case we would have $$2^{2i} +2^i + 1 = 
 a_0 (2^i+1) + a_1 (2^{i+1} +1) +
\cdots + a_i (2^{2i} + 1) \geq (a_0 + a_1+\cdots + a_i)(2^i + 1) \geq (2^i+1)^2 , $$
which is also impossible.   This shows $2^{2i} + 2^i + 1$  is
not representable.

We now argue that if $2^{2i} + 2^i + 1 < x \leq 2^{2i} + 2^{i+1} + 2$,
then $x$ has a representation.   This suffices to show that all
$x > 2^{2i} + 2^i + 1$ are representable, because this range contains
$2^i + 1$ consecutive integers, and any $x > 2^{2i} + 2^{i+1} + 2$
can then be represented by adding the appropriate multiple of $2^i + 1$.

Given a particular linear combination $$x = 
a_0 (2^i+1) + a_1 (2^{i+1} +1) +
\cdots + a_i (2^{2i} + 1),$$ 
call its {\it weight\ }  $a_0+\cdots+ a_i$.
We now
repeat the following transformation:  given a linear combination
$x = a_0 (2^i + 1) + \cdots + a_i(2^{2i} + 1)$, find
the largest nonzero $a_j$ in the combination.
Then form the linear combination of $x+1$ by adding $2\cdot (2^{j-1} + 1) 
- (2^j +1) = 1$ to the representation for $x$.
Doing so increases the weight of the
linear combination by $1$, because we add $2$ to one coefficient and
subtract $1$ from another.  

Now let us start the process with the number $(2^{2i}+1) + (2^i + 1)$,
which has a representation of weight $2$.
When we carry out the transformation of the previous paragraph once,
the $1$ coefficient of $2^{2i}+1$ in the linear combination
disappears and a $2$ appears as the coefficient of $2^{2i-1} + 1$.
Doing it twice more causes this $2$ to disppear, and a $4$
appears as the coefficient of $2^{2i-2} + 1$.   This process continues
for a total of $1+2+4+\cdots+ 2^{i-1} = 2^i - 1$ times, eventually
resulting in the representation $2^{2i} + 2^i + 2 + 2^i - 1 
= 2^{2i} + 2^{i+1} + 1 $ as $(2^i+1)(2^i + 1)$ of weight $2^i + 1$.
Finally, $2^{2i} + 2^{i+1} + 2$ has the representation
$(2^{2i}+1)+(2^{i+1} +1)$.
This gives us $2^i+1$ consecutive
representable numbers, as desired, and completes the proof.
\end{proof}

We have now shown that $G_{\bf s} (i) = 2^{2i} + 2^i + 1$
for $i \geq 1$.  Hence we get our desired counterexample:
the characteristic sequence of $(2^i + 1)_{i \geq 0}$ is
automatic, as the set of its base-$2$ representations
is specified by the regular expression
$10^*1$.
But the characteristic sequence of $(G_{\bf s} (i))_{i \geq 1} 
= (2^{2i} + 2^i + 1)_{i \geq 1}$  is not automatic, as
the set of its base-$2$ representations is of
the form $\{ 10^i 10^i 1 \, : \, i \geq 0\}$, which can easily
be seen to be non-regular using a standard tool from formal
language theory called the pumping lemma \cite[Lemma 3.1]{Hopcroft&Ullman:1979}.

\begin{remark}
Theorem~\ref{counter}, in more generality, also appears in
a recent paper of Song \cite{Song:2020}.   However, the proof was
omitted there, so we give it here.
\end{remark}

\section{Concluding remarks}

We conjecture that the analogue of Corollary~\ref{nine} holds for
all Beatty sequences.

For other results of additive number theory based on automata
theory, see \cite{Rajasekaran&Shallit&Smith:2020}.

All the {\tt Walnut} code we used is available from the author's
website, \url{https://cs.uwaterloo.ca/~shallit/papers.html}.


\begin{thebibliography}{99}

\bibitem{Allouche&Cloitre&Shevelev:2016} 
J.-P. Allouche, B. Cloitre, and V. Shevelev.
\newblock Beyond odious and evil.
\newblock {\it Aequationes Math.} {\bf 90} (2016), 341--353.

\bibitem{Allouche&Shallit:1999}
J.-P. Allouche and J.~Shallit.
\newblock The ubiquitous {Prouhet-Thue-Morse} sequence.
\newblock In C.~Ding, T.~Helleseth, and H.~Niederreiter, editors, {\em
  Sequences and Their Applications, Proceedings of SETA '98}, pp.  1--16.
  Springer-Verlag, 1999.

\bibitem{Allouche&Shallit:2003}
J.-P. Allouche and J.~Shallit.
\newblock {\it Automatic Sequences}.
\newblock Cambridge University Press, 2003.

\bibitem{Berlekamp&Conway&Guy:1982}
E. R. Berlekamp, J. H. Conway, and R. K. Guy.
\newblock {\it Winning Ways for your Mathematical Plays, Vol. 2: Games in Particular.}
\newblock Academic Press, 1982.

\bibitem{Carpi&Maggi:2001}
A. Carpi and C. Maggi.
\newblock On synchronized sequences and their separators.
\newblock {\it RAIRO Inform. Th\'eor. App.} {\bf 35} (2001), 513--524.

\bibitem{Charlier&Rampersad&Shallit:2012}
E.~Charlier, N.~Rampersad and J.~Shallit.
\newblock Enumeration and decidable properties of automatic sequences.
{\em Internat. J. Found. Comp. Sci.} {\bf 23} (2012), 1035--1066.

\bibitem{Dutch&Rickett:2012}
K. Dutch and C. Rickett.
\newblock Conductors for sets of large integer squares.
\newblock {\it Notes on Number Theory and Discrete Mathematics}
{\bf 18} (1) (2012), 16--21.

\bibitem{Erdos&Graham:1972}
P. Erd\H{o}s and R. L. Graham.
\newblock On a linear diophantine problem of Frobenius.
\newblock {\it Acta Arith.} {\bf 21} (1972), 399--408.

\bibitem{Hopcroft&Ullman:1979}
J. E. Hopcroft and J. D. Ullman.
\newblock {\it Introduction to Automata Theory, Languages, and 
Computation.}
\newblock Addison-Wesley, 1979.

\bibitem{Kawsumarng&Khemaratchatakumthorn&Noppakaew&Pongsriiam:2021}
S. Kawsumarng, T. Khemaratchatakumthorn, P. Noppakaew, and P. Pongsriiam.
\newblock Sumsets associated with Wythoff sequences and Fibonacci numbers.
\newblock {\it Period. Math. Hung.} {\bf 82} (2021), 98--113.

\bibitem{Lekkerkerker:1952}
C.~G. Lekkerkerker.
\newblock Voorstelling van natuurlijke getallen door een som van getallen van
  {Fibonacci}.
\newblock {\em Simon Stevin} {\bf 29} (1952), 190--195.

\bibitem{Moscariello:2015}
A. Moscariello.
\newblock On integers which are representable as 
sums of large squares.
\newblock {\it Intl. J. Number Theory} 
{\bf 11} (2015), 2505--2511.

\bibitem{Rajasekaran&Shallit&Smith:2020}
A. Rajasekaran, J. Shallit, and T. Smith.
Additive number theory via automata theory.
{\it Theor. Comput. Sys.} {\bf 64} (2020), 542--567.

\bibitem{RamirezAlfonsin96}
J.~L. Ram{\'i}rez~Alfons{\'i}n.
\newblock Complexity of the {F}robenius problem.
\newblock {\em Combinatorica} {\bf 16} (1996), 143--147.

\bibitem{RamirezAlfonsin05}
J.~L. Ram{\'i}rez~Alfons{\'i}n.
\newblock {\em The {D}iophantine {F}robenius Problem}.
\newblock Vol.~30 of
{\em Oxford Lecture Series in Mathematics and its Applications},
Oxford University Press, Oxford, 2005.

\bibitem{Reble:2008}
D.~Reble.
\newblock Zeckendorf vs.\ Wythoff representations: comments on 
\seqnum{A007895}.
\newblock Manuscript available at \url{https://oeis.org/A007895/a007895.pdf},
  2008.

\bibitem{Rosales&Garcia-Sanchez:2009}
J. C. Rosales and P. A. Garc{\'i}a-S{\'a}nchez.
\newblock {\it Numerical Semigroups.}
\newblock Springer, 2009.

\bibitem{Schaeffer&Shallit:2012}
L.~Schaeffer and J.~Shallit.
\newblock The critical exponent is computable for automatic sequences.
\newblock {\em Internat. J. Found. Comp. Sci.} {\bf 23} (2012), 1611--1626.

\bibitem{Shallit:2021}
J. Shallit.
\newblock Sumsets of Wythoff sequences, Fibonacci representation, and beyond.
\newblock To appear, {\it Period. Math. Hung.}, 2021.  
\newblock Preprint at \url{https://arxiv.org/abs/2006.04177}.

\bibitem{Sloane}
N. J. A. Sloane et al.
\newblock {\it The On-Line Encyclopedia of Integer Sequences}.
\newblock Available at \url{https://oeis.org}, 2021.

\bibitem{Song:2020}
K. Song.
\newblock The Frobenius problem for numerical semigroups generated by the
Thabit numbers of the first, second kind base $b$ and the 
Cunningham numbers.
\newblock {\it Bull. Korean Math. Soc.} {\bf 57} (2020), 623--647.

\bibitem{Zeckendorf:1972}
E.~Zeckendorf.
\newblock {Repr\'esentation} des nombres naturels par une somme de nombres de
  {Fibonacci} ou de nombres de {Lucas}.
\newblock {\em Bull. Soc. Roy. {Li\`ege}} {\bf 41} (1972), 179--182.


\end{thebibliography}
\end{document}